\documentclass[reqno,12pt]{amsart}

\usepackage{amsthm,amsmath,amsfonts,amssymb,amscd,mathrsfs,graphicx,cases,extarrows,indentfirst}

\usepackage{enumerate}
\usepackage{bbm} %\mathbb numbers and other symbols
\usepackage{hyperref,supertabular}
\usepackage{yfonts}
\usepackage{eucal}
\newcommand{\Vf}{\mathfrak{C}\kern-.05em\mathfrak{R}}

\DeclareMathOperator{\B}{\mathscr{B}\kern-.1em}

\DeclareMathOperator{\Jac}{Jac}

\DeclareMathOperator{\im}{im}

\newcommand{\Bchi}{\B\kern-.1em\chi}
\newcommand{\BchiR}{\B\kern-.1em\zeta_R}

\newcommand{\CC}{\mathbb C}
\newcommand{\Ccal}{{\mathcal{C}}}

\newcommand{\Dcal}{{\mathcal{D}}}

\newcommand{\Fcall}{\mathcal{F}}

\newcommand{\NN}{\mathbb N}

\newcommand{\OO}{\Ocal}
\newcommand{\Ocal}{\mathscr O}
\newcommand{\bD}{{\bar{D}}}

\newcommand{\cjcl}[1]{[ \kern-.15em [ #1 ]\kern-.15em ]}

\newcommand{\crst}{\mathscr{C \kern-.25em R}} %stack quot of critical locus in \X

\newcommand{\git}[1]{{/\!\!/_{\kern-.2em #1}}\,}

\DeclareMathOperator{\A}{\mathcal{A}}

\DeclareMathOperator{\bpat}{{\bar{\partial}}}

\newcommand{\fb}{{\bar{f}}}

\DeclareMathOperator{\Dom}{{Dom}}
\DeclareMathOperator{\Hcall}{\mathcal{H}}

\newcommand{\ib}{{\bar{i}}}

\newcommand{\jb}{{\bar{j}}}

\newcommand{\pat}{{\partial}}

\newtheorem{thm}{Theorem}[section] % number like 3.1, 3.2, 3.3, etc.
\newtheorem{pro}[thm]{Proposition} % numbered with thm
\newtheorem{lem}[thm]{Lemma} % numbered with thm
\newtheorem{cor}[thm]{Corollary} % numbered with thm
\theoremstyle{definition} % 'here we change the style'
\newtheorem{defn}[thm]{Definition} % numbered with thm
\theoremstyle{remark} % 'style changed again'
\newtheorem{rem}[thm]{Remark}% numbered with thm
\newtheorem{notn}[thm]{Notation}% numbered with thm
 \title[$\bpat_f$-Neumann problem]{A Twisted $\bpat_f$-Neumann Problem and Toeplitz $n$-tuples from singularity theory}
 \author{Hao Wen}
 \address{School of Mathematical Sciences, Peking University, Beijing, China}
 \email{jiange133@pku.edu.cn}
 \author{Huijun Fan$^\dag$}
\thanks{$\dag$ Supported by NSFC(11271028), NSFC(11325101), and Doctoral Fund of Ministry of Education of China(20120001110060)}
\address{School of Mathematical Sciences, Peking University, Beijing, China}
\email{fanhj@math.pku.edu.cn} 
\begin{document}
 \maketitle
\numberwithin{equation}{section}
 \begin{abstract}
     A twisted $\bpat_f$-Neumann problem associated to a singularity $(\OO_n, f)$ is established. By constructing the connection to the Koszul complex for toeplitz $n$-tuples $(f_1,\cdots,f_n)$ on Bergman spaces $B^0(D)$, we can solve this $\bpat_f$-Neumann problem. Moreover, the cohomology of the $L^2$ holomorphic Koszul complex $(B^*(D),\pat f\wedge )$ can be computed explicitly. 
 \end{abstract}
 %\tableofcontents

 \section{Introduction}
 Let $D$ be a bounded pseudoconvex domain in $\mathbb{C}^n$ with $C^\infty$ smooth boundary $\partial D$ and $f$ a holomorphic function on $\bar{D}$ with only isolated critical points in $D$ and no critical points on $\partial D$. Under such assumption, we get two objects in the framework of analysis.

 The first object is the toeplitz $n$-tuples with symbols $(f_1,f_2,\cdots,f_n)$ defined on the Bergman space on $D$, where the $f_i$'s are partial derivatives of $f$. One can study the $L^2$ holomorphic complex $(B^*(D),\pat f\wedge)$ given by
$$
 0\rightarrow B^0(D) \xlongrightarrow{\partial f \wedge} B^1(D) \xlongrightarrow{\partial f \wedge} \cdots \xlongrightarrow{\partial f \wedge} B^n(D) \rightarrow 0.
$$
Note that if without $L^2$ condition, this complex is an algebraic Koszul complex. If assuming $(f_1,\cdots,f_n)$ is regular, then the homology of the algebraic Koszul complex will only be nontrivial on the top term and is isomorphic to the Jacobian ring of $f$ on $D$. In the assumption of $L^2$ integrability, lack of noetherian ring structure make things complicated. This complex is an important example in Taylor's multivariable spectral theory(ref.\cite{Ta}) and has been studied a lot. The spectral picture, spectral mapping theorem and the index theory were all developed(ref.\cite{EP}).  The index of this complex is computed to be the dimension of $\Jac(f)$ on $D$(ref.\cite{EP}, Chapter 10). The fact that the cohomology is concentrated at the $n^{th}$ degree should be known (we were informed by M. Putinar \cite{Pu} that this can be proved via the spectral localization technique), but the direct proof seems not so easy. In this paper, we will reprove this result via the study of $\bpat_f$ operator.

On the other hand, we can define the twisted Cauchy-Riemann operator $\bar{\partial}_f := \bar{\partial} + \partial f\wedge$ on $D$, which only preserving the real grading of the differential forms, not the Hodge grading. This operator was used by physicists to study the topological field theory of Landau-Ginzburg model from the B side(ref. \cite{Ce,CV}). In recent years, LG model has been found to be a very important part of 2-d topological field theory, mirror symmetry and categorification theory of open strings(ref. \cite{FJR,CR,FJ},\cite{GMW},\cite{KKS}). Inspired by the physicists' work, the second author  proposed an approach (\cite{Fa}) to study the singularity theory of $f$ by constructing the Hodge theory for the operator $\bpat_f$ and the twisted Laplacian $\Delta_f=\bpat_f\bpat_f^*+\bpat_f^*\bpat_f$. The aim is to construct the Saito's Frobenius manifold structure (ref. \cite{ST}) for singularities and eventually treat the quantization problem of LG model from the B side. Recently, a different method via the theory of polyvector fields was built by Li-Li-Saito \cite{LLS} for studying the singularity and the related primitive forms, which however did not touch the Hodge structure. The paper \cite{Fa} can only treat the marginal deformation of a general singularity, but not  the universal deformation of a singularity. Hence to recover Saito's Frobenius manifold structure from the analytical method, we must study some boundary value problem of $\bpat_f$ operator.

In this paper, we will study the $\bpat_f$-Neumann problem on $D$. This problem is related to the $L^2$ complex $(L^2(D),\bpat_f)$, whose cohomology group is denoted by $H^*_{((2),\bpat_f)}$. As the first result, we can solve the $\bpat_f$-Neumann problem by proving the strong Hodge decomposition theorem as below.

\begin{thm}\label{sec1:theo-1}
Let $D$ be a bounded strongly pseudoconvex domain in $\mathbb{C}^n$ with $C^\infty$ smooth boundary $\partial D$ and $f$ a holomorphic function on $\bar{D}$ with only isolated critical points in $D$ and no critical points on $\partial D$. Then we have the decomposition
\begin{equation}
H^*(D)=\Hcall^*\oplus \im\bpat_f\oplus \im\bpat_f^*.
\end{equation}
and then the isomorphism
\begin{equation}
H^*_{((2),\bpat_f)}\cong \Hcall^*.
\end{equation}
Furthermore, all the spaces $\Hcall^*$ are of finite dimensional.
 \end{thm}
Theorem \ref{sec1:theo-1} is a direct conclusion of Theorem \ref{sec3:theo-4} and Corollary \ref{sec3:coro-1}. To prove this theorem, we first show that $\Delta_f$ with $\bpat_f$-Neumann boundary condition is a self-adjoint operator, hence there exists a weak Hodge decomposition. Usually, to prove the strong decomposition, we need a global a priori estimate for the Green operator which can naturally deduce the compactness by Rellich theorem. However, things are different in our $\bar\partial_f$-Neumann problem. The $\Delta_f$ operator always mix $(k,0)$-forms with other types of forms, thus the a priori estimate becomes complicated because there is no global estimate to control the Sobolev norms of holomorphic $k$ form in $\bpat$-Neumann problem. However, we have an indirect way to get around this problem. We can construct an isomorphism between the $L^2$ complex $(L^2(D),\bpat_f)$ and the $L^2$ holomorphic complex $(B^*(D), \pat f\wedge)$. By Taylor's joint spectral theory, the cohomology of the later complex can be proved to be of finite dimension. Using the finite dimensionality and a theorem from functional analysis, we can prove the range of $\bpat_f $ and $\bpat_f^*$ are all closed. So this proves the strong Hodge decomposition, meanwhile we can prove that the spectrum of $\Delta_f$ has a gap at $0$.

Conversely, by studying the complex $(L^2(D),\bpat_f)$ in $C^\infty$ category, we can calculate the cohomology of $(B^*(D),\pat f\wedge )$ as mentioned above. The second main result is as follows.

\begin{thm}\label{sec1:theo-2}
 If $D$ is a bounded strongly pseudoconvex domain in $\mathbb{C}^n$ with $C^\infty$ smooth boundary $\partial D$ and $f$ is a holomorphic function on $\bar{D}$ with isolated critical points in $D$ and no critical points on $\partial D$, then the dimension of the Koszul cohomology on Bergman spaces is concentrated at the $n^{th}$ degree and equal to the number of critical points, with multiplicities accounted, of $f$ in $D$.
 \end{thm}

\begin{cor} Under the assumption of Theorem \ref{sec1:theo-2}, The cohomology groups
$$
H^*_{\bpat_f}(D),H^*_{\bpat_f}(\bD),H^*_{(c,\bpat_f)},H^*_{((2),\bpat_f)}, H^*_{\bpat_f}(\Ccal),H^*_{(0,\bpat_f)}
$$
are all isomorphic to the space $\Hcall^*$.
\end{cor}

  \begin{rem}
 For arbitrary $n$-tuples $(f_1,f_2,\cdots,f_n)$ satisfying the condition that they have only finite common zeros and have no common zeros on $\partial D$, the proof in our article can be applied to the operator $\bar{\partial}+(f_1dz_1+ \cdots +f_ndz_n)\wedge$ and all our results still holds. In fact, throughout our article, we will not use the fact that the $f_i$'s are the partial derivatives of a single function.
 \end{rem}

\begin{notn} We use the super-bracket
$$[A,B]:= AB-(-1)^{\deg(A)\deg(B)}BA$$ in this paper.
 \end{notn}

 \section{$\bpat_f$-Neumann Problem}

 Let $h=\sum_i \frac{1}{2} dz^{i} \otimes d z^{\ib}$ be the standard hermitian metric of $\mathbb{C}^{n}$ in the coordinate system $\{z_i,i=1,\cdots,n\}$.  Let $D$ be a bounded strongly pseudoconvex domain in $\mathbb{C}^n$ with smooth boundary and $f$ a holomorphic function on $\bar{D}$.

The study of pseudoconvexity is one of the central topic in the theory of functions of several complex variables. $D$ is called  pseudoconvex if it can be exhausted by a continuous plurisubharmonic function. Every (geometrically) convex domain in $\CC^n$ is pseudoconvex. If the boundary $\pat D$ is $C^2$, then this is equivalent to the Levi pseudoconvexity we will explain below.

Let $r$ be a $C^2$ function defined in a neighborhood of $p\in \partial D$ satisfying $r|_{\partial D}=0$ and $\|dr\|=1$ on $ \partial D$. Then we can define a Levi form $L_p$ along the $n-1$ dimensional subspace $\{\xi\in T_p^\CC\CC^n|\sum_{j=1}^n\frac{\pat r}{\pat z^j}\xi^j=0\}$ by
\begin{equation}
 L_p(\xi,\eta)=\sum_{i,j}\frac{\pat^2 r}{\pat z^i \pat z^\jb}\xi^i\bar\eta^j.
\end{equation}
If the Levi form $L_p$ is semi-positive at all points $p\in \pat D$, then $D$ is said to be Levi pseudoconvex. If $L_p$ is positive at all points of $\pat D$, then $D$ is said to be strongly pseudoconvex. The balls in $\CC^n$ are strongly pseudoconvex.

Denote by $\A^{p}(D)$ (or $\A^{p,q}(D)$) the space of smooth $p$ forms (or $(p,q)$-forms) on $D$ and $\A(D)=\oplus_p \A^p(D)$. Let $\A^p(\bD)$ be the subspace of $\A^p(D)$ whose elements can be extended smoothly to a small neighborhood of $\bD$ and $\A(\bD)=\oplus_p \A^p(\bD)$. $\A_c^p(D)$ is a subspace of $\A^p(\bD)$ whose elements have compact support disjoint from $\pat D$. Similarly, we have the definitions of $\A^{p,q}(\bD)$ and $\A^{p,q}_c(\bD)$.

For any form $\varphi \in \A^{p,q}(\bD)$, we have the expression
$$
\varphi={\sum_{I,J}}^{\prime}\varphi_{I,J}dz^{I}\wedge d\bar z^J,
$$
where ${\sum}^{\prime}$ means summation over strictly increasing multi-indices and $\varphi_{I,J}$'s are antisymmetric for arbitrary $I$ and $J$.

For any $(p,q)$-forms $\varphi={\sum_{I,J}}^{\prime}\varphi_{I,J}dz^{I}\wedge d\bar z^J$ and $\psi={\sum_{I,J}}^{\prime}\psi_{I,J}dz^{I}\wedge d\bar z^J$, we can define the $L^2$ inner product:

 $$
 \langle \varphi,\psi \rangle={\sum_{I,J}}^{\prime}\langle \varphi_{I,J},\psi_{I,J}\rangle={\sum_{I,J}}^{\prime}\int_{D}\varphi_{I,J}\overline{\psi_{I,J}}dV
 $$
 where $dV$ denote the volume element on $D$ defined by $h$.

 Let $\|\centerdot\|$ be the corresponding $L^2$-norm and $L^2_{(p,q)}(D)$ be the $L^2$-completion space of $\A^{p,q}(\bar D)$. Define $L^2_k(D)=\oplus_{p+q=k}L^2_{(p,q)}(D)$ and $L^2(D)=\oplus_kL^2_k(D)$. Furthermore, the Sobolev $s$-norms $\|\centerdot\|_{s}$ and the corresponding Sobolev spaces $W^s_{(p,q)}(D),W^s_k(D), W^s(D)$ can be defined. For example, for non-negative integer $s$, elements of $W^s(D)$ has derivatives in $L^2(D)$ up to $s$ order and $\|\varphi\|_{s}$ is the sum of $L^2$ norms of derivatives of $\varphi$ up to $s$ order. In particular, we have $W^0(D)=L^2(D)$.

 Now any differential operator $T$ defined on $\A(\bD)$ can be extended to a unbounded closed operator in $L^2(D)$ by means of generalized derivatives. Remember that if $T$ is a closed operator defined on $\Dom(T)\subset L^2(D)$ if and only if the following holds: if $\varphi_i \in \A^p(\bD) \cap L^2(D)$ and $ T(\varphi_i) \in L^2(D)$ are function sequences such that $\varphi_i \rightarrow \varphi$ and $T(\varphi_i) \rightarrow \psi \in L^2(D)$, then  $\varphi$ is in ${\Dom}(T)$ and $T(\varphi)= \psi$.

Now the Cauchy-Riemann operator $\bpat$ and the twisted operator operator $\bpat_f=\bpat+\pat f\wedge:\A^k(\bD)\to\A^k(\bD)$ can be extended to closed operators in $L^2(D)$ such that
 \begin{align*}
{\Dom}(\bar{\partial})&=\{\varphi \in L^2(D)|\bar{\partial} \varphi \in L^2(D)\}\\
{\Dom}(\bar{\partial}_f)&=\{\varphi \in L^2(D)|\bar{\partial}_f \varphi \in L^2(D)\}.
\end{align*}

Since $f$ is bounded on $D$, the multiplication operator $\pat f\wedge$ has the domain ${\Dom}(\partial f \wedge)=L^2(D)$ and actually we have
 $$
 {\Dom}(\bar{\partial}_f)= {\Dom}(\bar{\partial}).
$$

 Now we consider the adjoint of  $\bar{\partial}_{f}$ under the $L^2$ norm. By definition the Hilbert space adjoint $\bar{\partial}_{f}^{*}$ of $\bar{\partial}_{f}$ is defined on the domain  ${\Dom}(\bar{\partial}_{f}^{*})$ consisting of all $\varphi \in L^2_k(D)$ such that
 $|\langle \varphi, \bar{\partial}_{f}(\psi) \rangle| \leq c||\psi||$ for some positive constant $c$ and for all $\psi \in {\Dom}(\bar{\partial}_f)$.
 As $\partial f\wedge$ is bounded, the above inequality is equivalent to $|\langle \varphi, \bar{\partial}(\psi) \rangle| \leq c'||\psi||$ for some positive constant $c'$.
 This means ${\Dom}(\bar{\partial}^*_f)= {\Dom}(\bar{\partial}^*)$ and they have the same Neumann boundary conditions.

For $\varphi, \psi \in \A(\bar D)$, we have the integration by parts:
\begin{align}
 \langle \bar{\partial}_f \varphi, \psi \rangle&= \langle \varphi, \vartheta_f \psi \rangle + \int_{\partial D} \langle \sigma(\bar{\partial}, dr)\varphi, \psi \rangle\\
 \langle \vartheta_f \varphi, \psi \rangle&= \langle \varphi, \bar{\partial}_f \psi \rangle + \int_{\partial D} \langle \sigma(\vartheta, dr)\varphi, \psi \rangle.
 \end{align}
 Here
 $$
 \vartheta_f=\vartheta+\fb_j\iota_{\pat_j},
 $$
 where $\vartheta$ represents the formal adjoint of $\bpat$ and $\iota_{\pat_{\jb}}$ is the contraction operator with the vector $\pat_\jb=\frac{\pat}{\pat \bar z^j}$, and
 $$
 \sigma(\bar{\partial}, dr)=\bpat r\wedge=\frac{\pat r}{\pat \bar z^j}d\bar z^j\wedge,\;\sigma(\vartheta, dr)=-\frac{\pat r}{\pat z_j}\iota_{\pat_{\jb}}.
$$
Hence we have
 $$
 {\Dom}(\bar{\partial}_{f}^{*}) \cap \A(\bar{D})=\{\varphi \in \A(\bar{D})|\sigma(\vartheta, dr)\varphi=0 \; on \; \partial D\}
 $$
Denote by $\Dcal^{p,q}=\{\varphi \in \A(\bar{D})|\sigma(\vartheta, dr)\varphi=0 \; on \; \partial D\}$ and $\Dcal^k=\oplus_{p+q=k}\Dcal^{p,q}$.

\begin{defn}Let
$\Delta_{f}=[\bar{\partial}_f,\bar{\partial}_{f}^{*}]=\bar{\partial}_f\bar{\partial}_{f}^{*}+\bar{\partial}_{f}^{*}\bar{\partial}_f
$ be the operator from $L^2(D)$ to $L^2(D)$ with domain $
{\Dom}(\Delta_f)=\{\varphi \in L^2(D)|\varphi \in {\Dom}(\bar{\partial}_f) \cap {\Dom}(\bar{\partial}_{f}^*); \bar{\partial}_f(\varphi) \in {\Dom}(\bar{\partial}_{f}^{*}) \; and \; \bar{\partial}_{f}^{*}(\varphi) \in {\Dom}(\bar{\partial}_f)\}$.
\end{defn}

\begin{pro}
 $\Delta_f$ is a linear, densely defined, closed self-adjoint operator.
 \end{pro}

\begin{proof}The proof is the same as the proof of Proposition 1.3.8 in \cite{FK}.

\end{proof}

\begin{rem} We can consider the formal Laplacian $\hat{\Delta}_f=\bpat_f\vartheta_f+\vartheta_f\bpat_f+I$ defined on $\Dcal^{p,q}$. This operator has a unique Friedrichs self-adjoint extension related to the quadratic form $Q(\varphi,\phi)=(\bpat_f\varphi,\bpat_f\psi)+(\vartheta_f\varphi,\vartheta_f\psi)+(\varphi,\psi)$. This extended self-adjoint operator is just $\Delta_f+I$ and the equivalence relation is clear by the standard abstract theorem in functional analysis.

The self-adjointness of $\Delta_f$ is due to the $\bpat$-Neumann boundary condition which is characterized by
 \begin{equation}\label{sec2:Neum-bdry-1}
 \begin{aligned}
 {\Dom}(\Delta_f) \cap \A(\bar{D})=\{\varphi \in \A(\bar{D})|&\sigma(\vartheta, dr)\varphi=0 \;\text{and}\; \\
 &\sigma(\vartheta, dr)\bar{\partial} \varphi=0\;\text{on}\;\partial D\}.
 \end{aligned}
\end{equation}
\end{rem}
Similar to the $\bpat$-Neumann problem, here we want to solve the equation $\Delta_f \varphi=\eta\in L^2(D)$ under the $\bpat$-Neumann boundary condition. We call this as {\em $\bpat_f$-Neumann problem}.

Since $\Delta_f$ is self-adjoint and $Im(\bar{\partial}_f) \perp Im(\bar{\partial}^*_f)$, we get a \emph{weak} Hodge decomposition
 \begin{equation}
 L^2_k(D)={\Hcall}^k \oplus \overline{Im(\Delta_f)}={\Hcall}^k \oplus \overline{Im(\bar{\partial}_f)} \oplus \overline{Im(\bar{\partial}^*_f)}
 \end{equation}
 where ${\Hcall}^k$ denote the kernel of $\Delta_f$.

 To solve the $\bpat_f$-Neumann problem, we need to prove that all the range in the above decomposition are closed. The $\bpat_f$-Neumann problem will display different nature compared to the $\bpat$-Neumann problem, in which $f$ will play dominant role. This will be shown in next section.

 \section{$\bpat_f$-complexes, finite dimensionality and spectral gap}

 In this section, we first discuss various $\bpat_f$ complexes defined on a bounded pseudoconvex domain. Then we will show that the $L^2$ $\bpat_f$-complex has finite dimensional cohomology groups and there exists a spectral gap between $0$ and other spectra of $\Delta_f$. In the $\bar{\partial}$-Neumann problem, there is no estimate for $L^2$ integrable holomorphic $(p,0)$-forms, which is in the kernel of $\Delta_{\bar\partial}$, near the boundary. For this reason, we solve the $\bar{\partial}_f$-Neumann problem in a indirect way.  We avoid to estimate directly the behavior of the operator $\Delta_f$, which twist the $(p,0)$-forms and other types of forms, instead, we will use a classical result in multivariable spectra theory about the complex $(B^*(D),\pat f\wedge)$ and some results in the theory of unbounded linear operators.

 \subsection{$\bpat_f$-complexes}\

 There are various $\bpat_f$-complexes which are defined by smoothness or boundary value conditions. At first, we have the $L^2$ $\bpat_f$-complex
\begin{equation}
L^2(D): L^2_0(D) \xlongrightarrow{\bar{\partial}_f} L^2_1(D) \xlongrightarrow{\bar{\partial}_f} L^2_2(D) \xlongrightarrow{\bar{\partial}_f} \cdots
\end{equation}
 corresponding to $L^2$ integrable $p$-forms. The cohomology group is defined as
 $$
 H^k_{((2),\bpat_f)}=\frac{\{\varphi\in \Dom(\bpat_f)|\bpat_f\varphi=0\}}{\bpat_f (\Dom(\bpat_f))}
 $$

In addition, there are $\bpat_f$-complexes $\A^*(D),\A^*(\bD),\A^*_c(D)$, which correspond to smooth $p$-form on $D$, on $\bD$, and having compact support in $D$ respectively. We denote by $H^k_{\bpat_f}(\bD),H^k_{\bpat_f}(D),H^k_{(c,\bpat_f)}(D)$ the corresponding cohomology groups.

Let $\Ccal^{p,q}=\{\varphi \in \A(\bar{D})|\sigma(\bpat, dr)\varphi=0 \; on \; \partial D\}$ and $\Ccal^k=\oplus_{p+q=k}\Ccal^{p,q}$. We can take $\vartheta_f$ as a closed operator in $H(D)$ at first and then consider the $\vartheta_f$-Neumann problem, and in this case, we have $\Ccal^{k}=\A^{k}(\bD)\cap \Dom(\vartheta_f^*)$.

\begin{lem}
$\bpat_f \Ccal^k\subset \Ccal^{k+1}$.
 \end{lem}

\begin{proof}If $\psi\in \Ccal^k$, then it can be written as $\psi=\bpat r\wedge \alpha+r\beta$ for $\alpha\in \A^{k-1}(\bD), \beta\in \A^k(\bD)$. Then $\bpat_f \psi=\bpat r\wedge (\beta-\bpat_f\alpha)+r\bpat_f \beta$ which is in $\Ccal^k$.
\end{proof}

This lemma shows that $(\Ccal^*,\bpat_f)$ forms a complex and has cohomology $H^k_{\bpat_f}(\Ccal)$.

As in \cite{FK}, we also have the Dirichlet or zero-boundary value cohomology
\begin{equation}
H^k_{(0,\bpat_f)}=\frac{\{\psi\in \A^k(\bD)|\bpat_f\psi=0,\psi|_{\pat D}=0 \}}{\bpat_f\{\psi\in A^{k-1}(\bD)|\psi|_{\pat D}=0,\bpat_f\psi|_{\pat D}=0\}}.
\end{equation}

\begin{pro}\label{sec3:prop-2} There exists isomorphism
$i:H^k_{(0,\bpat_f)}\cong H^k_{\bpat_f}(\Ccal)$.
 \end{pro}

 \begin{proof} Suppose that $\phi\in \A^k(\bD),\phi|_{\pat D}=0$, and $\phi=\bpat_f \psi$ with $\psi\in \Ccal^{k-1}$.
 Then $\psi$ has the form $\psi=\bpat r\wedge \alpha+r\beta$. This can be rewritten as
 $$
 \psi=\bpat_f(r\alpha)+r(-\bpat_f\alpha+\beta).
 $$
 Let $\psi_0=r(-\bpat_f\alpha+\beta)$. This gives $\phi=\bpat_f \psi_0$, which shows that $i$ is a well-defined injective map. To prove the surjectivity, suppose $\phi\in \Ccal^k$ and $\bpat_f\phi=0$. Then $\phi$ also has the expression $\phi=\bpat_f(r\alpha)+r(-\bpat_f\alpha+\beta)$. Hence $\phi$ is cohomologous to $r(-\bpat_f\alpha+\beta)$, which vanishes on $\pat D$.
 \end{proof}

 We will discuss other relations between these cohomologies in the following sections. Above all, we want to discuss the relation between the $L^2$ complex $(L^2(D),\bpat_f)$ and the $L^2$ holomorphic Koszul complex $(B^*(D),\pat f\wedge)$.

 \subsection{Koszul complex, finite dimensionality and spectral gap}\

 Let $B^k(D)$ be the $L^2$ integrable holomorphic $k$-form on $D$, i.e., $B^0(D)$ is the Bergman space on $D$ and $B^k(D)$ can be viewed as direct products of $B^0(D)$. The complex $(B^*(D),\pat f\wedge)$ is defined as
 $$0 \rightarrow B^0(D) \xlongrightarrow{\partial f \wedge} B^1(D) \xlongrightarrow{\partial f \wedge} \cdots \xlongrightarrow{\partial f \wedge} B^n(D) \rightarrow 0,$$
 whose cohomology are denoted by $H^*_{\partial f \wedge}(D)$.

 In 1970, J. L. Taylor ~\cite{Ta} developed a multivariable (joint) spectral theory.
 Given a Hilbert space X and a commuting $n$-tuples of bounded linear operators $T=(T_1,\cdots,T_n)$ on X, the joint spectra $\sigma(T,X)$ is the set of all $\lambda=(\lambda_1,\cdots,\lambda_n) \in \mathbb{C}^n$ such that $K^*(T-\lambda, B(D))$ is not acyclic.
 The essential joint spectra $\sigma_e(T,X)$ is the set of all $\lambda$ such that the cohomology of $K^*(T-\lambda, B(D))$ is not finite dimensional.
 The finite complex $K^*(T-\lambda, B(D))$ consists of the spaces
 $$K^p(T-\lambda,X)= X \otimes_{\mathbb{C}} \Lambda^p(\mathbb{C}^n) \quad (0 \leq p \leq n)$$
 and the coboundary operators
 $$d^p:K^p(T-\lambda,X) \rightarrow K^{p+1}(T-\lambda,X), \quad d^p(\varphi)= \tau \wedge \varphi$$
 where $\tau =(T_1-\lambda_1) \otimes e_1 + (T_2-\lambda_2) \otimes e_2 + \cdots (T_n-\lambda_n) \otimes e_n$ and $(e_1,e_2,\cdots,e_n)$ is the canonical basis of $\mathbb{C}^n$.

 The $L^2$ $\partial f \wedge$-complex $(B^*(D),\pat f\wedge)$ can be viewed as a model for Taylor's joint spectral theory.
 The Bergman space is a Hilbert space and the toeplitz operators defined by multiplication by $f_i= \frac{\partial f}{\partial z_i}, 1 \leq i \leq n$, is a commuting $n$-tuples of bounded linear operators.
 The $dz_i$'s can be viewed as a basis of $\mathbb{C}^n$.
 Thus the associated Koszul complex is exactly $(B^*(D),\partial f \wedge)$.

 Under our assumption, $D$ is bounded and pseudoconvex.
 By Theorem 8.1.1 and corollary 8.1.2 of  ~\cite{EP}, we have
 $$\sigma(z_1,\cdots,z_n,B^0(D))=\bar D$$
 and
 $$\sigma_e(z_1,\cdots,z_n,B^0(D))\subset \partial D.$$
 Furthermore by Theorem 8.2.1 and Proposition 8.2.5 of ~\cite{EP}, we have
 $$\sigma(f_1,\cdots,f_n, B^0(D))= \overline{(f_1,\cdots,f_n)(D)}$$
 and
 \begin{align*}
 \sigma_e(f_1,\cdots,f_n, B^0(D))&= (f_1,\cdots,f_n)(\sigma_e(z_1,\cdots,z_n,B^0(D)))\\
 &\subset(f_1,\cdots,f_n)(\partial D).
 \end{align*}

 Hence we have the simple conclusion:

 \begin{pro}\label{sec3:prop-1}
 Assume that $f$ is holomorphic on $\bD$ and has no critical points on $\pat D$, then
 $$
 0 \notin \sigma_e(f_1,\cdots,f_n, B^0(D)),
 $$
 which says that the complex $(B^*(D),\pat f\wedge)$ has at most finite dimensional cohomology group.
\end{pro}

 Now we turn to the discussion of the $L^2$ complex $(H^*(D),\bpat_f)$. The key theorem in this section is as follows.

 \begin{thm}\label{sec3:theo-1}
 There exists a quasisomorphism between the $L^2$ complex $(L^2(D),\bpat_f)$ and the complex $(B^*(D),\pat f\wedge)$.
 Moreover, their $p^{th}$ cohomology group vanishes for $n< p \leq 2n$.
 \end{thm}

 We are working in $L^2$ integrable category, so we must be careful to control the norms. Before proving Theorem \ref{sec3:theo-1}, we need the $L^2$ existence theorem for $\bar{\partial}$-Neumann problem \footnote{The assumption $n\geq 2$ doesn't matter in our references. This is because when $n=1$, the $\bar\partial$-Neumann condition is exactly the $0$-value Dirichlet condition for $(p,1)$-forms and all the existence and regularity theorems clearly hold by standard elliptic estimate.}.

\begin{thm}[\cite{Sh}]\label{sec3:theo-2}
 Let $D$ be a bounded pseudoconvex domain in $\mathbb{C}^{n}$ with $C^\infty$ smooth boundary.
 If $\bar{\partial} \varphi =0$ for some $\varphi \in L^2_{(p,q+1)}(D)$, then there exists $\psi \in L^2_{(p,q)}(D)$ such that $\varphi =\bar{\partial}\psi$ and $||\psi|| \leq c||\varphi||$.
 Here $0 \leq p \leq n$, $0 \leq q \leq n-1$ and $c$ is independent of the choice of $\varphi$.
 \end{thm}

 We will also need the Banach's closed range theorem as below.
 \begin{thm}[\cite{Sh}]\label{sec3:theo-3}
 Let $T: X\rightarrow Y$ be a closed linear operator between two Hilbert spaces and $T'$ be the transpose of $T$. Then the following conditions are equivalent:
 \begin{enumerate}
 \item $T$ has closed range in $Y$.
 \item $T'$ has closed range in $X$.
 \item There exists positive constant $c$, such that $$||Tx|| \geq c||x||, \forall x \in {\Dom}(T) \cap Ker(T)^{\bot}$$
 \item There exists positive constant $c$, such that $$||T'y|| \geq c||y||, \forall y \in {\Dom}(T') \cap Ker(T')^{\bot}$$
 \end{enumerate}
 \end{thm}

 \begin{proof}[ Proof of Theorem \ref{sec3:theo-1}]

 Assume $\bar{\partial}_f \varphi =0$ for some $\varphi \in L^2_p(D)$.
 To avoid too heavy notation, here and below we will use $a \lesssim b$ to denote 'there exists a constant $c>0$ such that $a \leq c \cdot b$'.

 Firstly we assume $n< p \leq 2n$.
 Let
 $$\varphi= \varphi^{n,p-n}+ \varphi^{n-1,p-n+1}+ \cdots + \varphi^{p-n,n}$$
 Then we have $\bar {\partial} \varphi^{p-n,n}=0$.
 By Theorem \ref{sec3:theo-2}, there exists $\psi^{p-n,n-1}$ such that $$\bar{\partial} \psi^{p-n,n-1}= \varphi^{p-n,n}$$ and $$||\psi^{p-n,n-1}||\lesssim ||\varphi^{p-n,n}||$$
 Then
 \begin{eqnarray*}
 \partial f \wedge \varphi^{p-n,n}+ \bar{\partial} \varphi^{p-n+1,n-1}
  &= &\bar{\partial} (\varphi^{p-n+1,n-1}- \partial f \wedge \psi^{p-n,n-1})\\
  &= &0
 \end{eqnarray*}
 Again by Theorem \ref{sec3:theo-2}, there exists $\psi^{p-n+1,n-2}$ such that $$\bar{\partial} \psi^{p-n+1,n-2}= \varphi^{p-n+1,n-1}- \partial f \wedge \psi^{p-n,n-1}$$ and
 \begin{eqnarray*}
 ||\psi^{p-n+1,n-2}|| &\lesssim & ||\varphi^{p-n+1,n-1}- \partial f \wedge \psi^{p-n,n-1}||\\
 &\lesssim & ||\varphi^{p-n+1,n-1}||+ ||\psi^{p-n,n-1}||\\
 &\lesssim & ||\varphi^{p-n+1,n-1}||+ ||\varphi^{p-n,n}||
 \end{eqnarray*}

 Inductively, we have $\psi^{p-n+k,n-1-k} \in L^2_{(p-n+k,n-1-k)}(D)$ such that
 $$\bar{\partial} \psi^{p-n+k,n-1-k}= \varphi^{p-n+k,n-k}- \partial f \wedge \psi^{p-n+k-1,n-k}$$ and
 $$||\psi^{p-n+k,n-1-k}|| \lesssim \sum_{i=0}^{k}||\varphi^{p-n+i,n-i}||$$
 Note that $\psi^{p-n-1,n}=0$ here.
 Let $\psi= \sum_{k=0}^{2n-p} \psi^{p-n+k,n-1-k}$, then $\bar{\partial}_f \psi= \varphi$ and $||\psi||\lesssim ||\varphi||$.
 This means that the $L^2$ $\bar{\partial}_f$-complex is exact at $p^{th}$ degree and thus the $p^{th}$ cohomology vanish for $n < p \leq 2n$.
 Moreover, by Theorem \ref{sec3:theo-3}, $\bar{\partial}_f$ has closed range at these degrees.

 Then we consider the case $0 \leq p \leq n$.
 Let
 $$\varphi= \varphi^{p,0}+ \varphi^{p-1,1}+ \cdots + \varphi^{0,p}$$
 Because $\bar{\partial}_f \varphi =0$, we have
 $$\bar{\partial} \varphi^{p-k,k}+ \partial f \wedge \varphi^{p-k-1,k+1}=0, 0 \leq k \leq p$$
 Similar to the discussion above, we have $\psi^{p-1-k,k}, 0 \leq k \leq p-1$ such that
 $$\bar{\partial} \psi^{p-k,k-1}= \varphi^{p-k,k}- \partial f \wedge \psi^{p-1-k,k}, 1 \leq k \leq p-1$$
 and
 \begin{eqnarray}
 \bar{\partial} (\varphi^{p,0}- \partial f \wedge \psi^{p-1,0})&=&0\\
 \partial f \wedge (\varphi^{p,0}- \partial f \wedge \psi^{p-1,0})&=& \partial f \wedge \varphi^{p,0}=0
 \end{eqnarray}
 Let $\psi= \sum_{k=0}^{p-1} \psi^{p-1-k,k}$, then $$\varphi= \bar{\partial}_f \psi + (\varphi^{p,0}- \partial f \wedge \psi^{p-1,0})$$
 Similar to the discussion above, norm of $\varphi^{p,0}- \partial f \wedge \psi^{p-1,0}$ can be controlled by that of $\varphi$ and $\psi$ and eventually by $\varphi$.
 Thus by (3.3) and (3.4), $\varphi^{p,0}- \partial f \wedge \psi^{p-1,0}$ is in $B^{p}(D)$ and represents an element in $H^p_{\partial f \wedge}(D)$.

 We define a map between the two complexes at the level of $L^2$ cohomology by:
 $$u: H^p_{\bar{\partial}_f}(D)\rightarrow H^p_{\partial f \wedge}(D), [\varphi]\mapsto [\varphi^{p,0}- \partial f \wedge \psi^{p-1,0}]$$
 If $[\varphi]= 0 \in H^p_{\bar{\partial}_f}(D)$, then we have $\varphi= \bar{\partial}_f \eta$,
 together with
 $$\varphi= \bar{\partial}_f \psi + (\varphi^{p,0}- \partial f \wedge \psi^{p-1,0})$$
 we have
 $$\varphi^{p,0}- \partial f \wedge \psi^{p-1,0}=\bar{\partial}_f (\eta-\psi)$$
 by counting degrees, we have
 $$\varphi^{p,0}- \partial f \wedge \psi^{p-1,0}=\partial f \wedge (\eta-\psi)^{p-1,0}$$
 i.e. $[\varphi^{p,0}- \partial f \wedge \psi^{p-1,0}] =0 \in H^p_{\partial f \wedge}(D)$ and thus $u$ is well defined.\\
 If $\eta \in B^p(D)$ represent a cohomology class, we have $\bar{\partial}_f \eta= 0$ and $u([\eta])= [\eta]$ and $u$ is surjective.\\
 If $u([\varphi])=[\varphi^{p,0}- \partial f \wedge \psi^{p-1,0}]=0 \in H^p_{\partial f \wedge}(D)$, then
 \begin{eqnarray*}
 \varphi&=& \bar{\partial}_f \psi + (\varphi^{p,0}- \partial f \wedge \psi^{p-1,0})\\
 &=& \bar{\partial}_f \psi+ \partial f \wedge \theta\\
 &=& \bar{\partial}_f (\psi+ \theta)
 \end{eqnarray*}
 i.e. $[\varphi]=0 \in  H^p_{\bar{\partial}_f}(D)$, and $u$ is injective.
 Thus $u$ is an isomorphism of cohomology groups.
 \end{proof}

 Now by the weak Hodge decomposition
 $$L^2_k(D) = {\Hcall}^k \oplus \overline{Im(\bar{\partial}_f)} \oplus \overline{Im(\bar{\partial}^*_f)}$$
 we have
 $$
 H^*_{((2),\bar{\partial}_f)}(D)=Ker(\bar{\partial}_f)/Im(\bar{\partial}_f) \cong {\Hcall}^* \oplus \overline{Im(\bar{\partial}_f)} / Im(\bar{\partial}_f),
 $$
 which is finite dimensional by Theorem \ref{sec3:theo-1}. By Corollary IV.1.13 of \cite{Go}: if a closed operator from a Banach space to another Banach space has finite cokernel, it must have closed range; we can conclude that $Im(\bar{\partial}_f)$ is closed. Now by Theorem \ref{sec3:theo-3}, $Im(\bar{\partial}^*_f)$ is also closed.
 Thus we have
 \begin{thm}[strong Hodge decomposition]\label{sec3:theo-4}
 \begin{equation}
 L^2_k(D) = {\Hcall}^k \oplus  {Im(\bar{\partial}_f)} \oplus  {Im(\bar{\partial}^*_f)}
 \end{equation}
 and the isomorphism
 \begin{equation}
 H^k_{((2),\bar{\partial}_f)}(D) \cong {\Hcall}^k
 \end{equation}
 \end{thm}

 For any
 $$\psi \in ({\Hcall}^*)^{\bot}=Ker(\bar\partial_f)^{\bot} + Ker(\bar\partial_f^*)^{\bot},$$
 let $\psi=\psi_1+\psi_2+\psi_3$ be the orthogonal decomposition of $\psi$ into
 $Ker(\bar\partial_f)^{\bot} \cap Ker(\bar\partial_f^*)$, $Ker(\bar\partial_f)^{\bot} \cap Ker(\bar\partial_f^*)^{\bot}$ and
 $Ker(\bar\partial_f) \cap Ker(\bar\partial_f^*)^{\bot}$.
 By Theorem \ref{sec3:theo-3} and closeness of range of $\bar\partial_f$ and $\bar\partial_f^*$,
 we then have
 \begin{eqnarray*}
 \langle \Delta_f \psi, \psi \rangle &=& ||\bar\partial_f(\psi)||^2+ ||\bar\partial_f^*(\psi)||^2\\
 &=& ||\bar\partial_f(\psi_1+\psi_2)||^2+ ||\bar\partial_f^*(\psi_2+\psi_3)||^2\\
 &\geq& c||\psi_1+\psi_2||^2+c||\psi_2+\psi_3||^2\\
 &=& c(||\psi_1||^2+2||\psi_2||^2+||\psi_3||^2)\\
 &\geq& c||\psi||^2
 \end{eqnarray*}
 for some positive constant $c$.
 That's, a spectral gap exists between $0$ and other spectra of $\Delta_f$.
 By Proposition \ref{sec3:theo-3}, the range of $\Delta_f$ is closed. So we obtain
 \begin{cor}\label{sec3:coro-1} Let $D$ be a bounded strongly pseudoconvex domain in $\mathbb{C}^{n}$ with $C^\infty$ smooth boundary. Let $f$ be a holomorphic function in $D$ without critical points on the boundary $\pat D$. Then the twisted Laplacian $\Delta_f$ has finite dimensional kernel and there exists a spectral gap between $0$ and other spectra of $\Delta_f$. The complex $(L^2(D),\bpat_f)$ has finite dimensional cohomology for $0 \leq p \leq n$ and zero cohomology for $n < p \leq 2n$.
 \end{cor}

Now Theorem \ref{sec3:theo-4} and Corollary \ref{sec3:coro-1} gives Theorem \ref{sec1:theo-1}.

\begin{notn}
Similar to the equality $I=\Delta N+P$ for the $\bar\partial$-Neumann problem, let $P:L^2(D)\to \Hcall^*(D)$ be the projection operator and $G:L^2(D)\to \Dom(\Delta_f)$ be the Green operator, we have the decomposition
$$I=\Delta_f G+P.$$
\end{notn}

 \section{Global regularity}

 The operator $\Delta_f$ has the following expansion:
 \begin{align*}
 \Delta_f=\Delta_{\bpat} + \sum_{i,j}(f_{ij}(d\bar{z}_j \wedge)^* dz_i \wedge + \bar{f}_{ij}d\bar{z}_i \wedge (dz_j \wedge)^*) +\sum^n_{i=1}|f_i|^2.
 \end{align*}
 Here $\Delta_{\bpat}=\bpat^*\bpat+\bpat\bpat^*$ and the last two summands, denoted by $L_f$ and $|\nabla f|^2$ respectively, are of order 0. Hence $\Delta_f$ is an elliptic operator of second order and has the interior regularity estimate. In the following, we say that an operator $T$ is globally regular if and only if it preserves $\A(\bar D)$. $T$ is exactly regular if it maps $W^s(D)$ continuously to itself for any non-negative integers. Exact regularity \emph{a priori} means global regularity by Sobolev imbedding theorem.

 To obtain a global estimate, we need sharper estimate about $\bar \partial$-Neumann operator on strongly pseudoconvex domains.
 We state the results needed in our proof in the following.

 \begin{thm}\label{sec5:thm-1}
 For the $\bar\partial$-Neumann problem on a strongly pseudoconvex domain $D$ with smooth boundary, let $P$ be the projection to space of harmonic forms and $N_{(p,q)}$ be the Neumann operator on $(p,q)$-forms. Then there exists positive constants $c_s$ depends only on $s$ such that the following global estimates hold:
 \begin{enumerate}
 \item $P$ is exactly regular, i.e. it maps $W^s(D)$ continuously to itself.
 \item $||N_{(0,0)}\psi||_{s+\frac{1}{2}} \leq c_s||\psi||_s$.
 \item $||\bar\partial^* N_{(p,q)} \psi||_{s+\frac{1}{2}}+||\bar\partial N_{(p,q)} \psi||_{s+\frac{1}{2}} \leq c_s||\psi||_s$ for $q \geq 1$.
 \item $||N_{(p,q)}\psi||_{s+1} \leq c_s||\psi||_s$ for all $q \geq 1$.
 \end{enumerate}
 \end{thm}
 
 \begin{proof}
 All of them are classical results for $\bar\partial$-Neumann problem. For (1), see comments behind Corollary 5.2.7 and Theorem 6.2.2 in \cite{Sh}. (3) and (4) is exactly Theorem 5.3.10 in \cite{Sh}. For (2), by Equality 5.3.34 in \cite{Sh},(3) and(4)
 $$||N_{(0,0)}\psi||_{s+\frac{1}{2}} \lesssim ||N_{(0,1)}\bar\partial \psi||_s \lesssim ||\bar\partial \psi||_{s-1} \lesssim ||\psi||_s.$$
 \end{proof}

 \begin{lem}\label{sec5:lem-1}
 Assume $D$ is strongly pseudoconvex with smooth $\partial D$ and $f$ has no critical points on $\partial D$. $u$ is a function in ${\Dom}(\Delta_f)$.
 If $$\Delta u+ |\nabla f|^2 u=g$$ for some function $g \in W^s_0(D)$, then $u \in W^s_0(D)$.
 \end{lem}
 Notice that the $0^{th}$ cohomology $H^0_{((2),\bar\partial f)}$ is obviously zero, so $0$ is not in the spectrum of $\Delta_f$. Therefore by Corollary \ref{sec3:coro-1}, the spectrum of $\Delta_f=\Delta + |\nabla f|^2$  has a positive lower bound and $\Delta + |\nabla f|^2$ has a bounded inverse $G_0$ with $||G_0 g|| \leq c||g||$.
 
 \begin{proof}
 By definition, $u \in L^2_0(D)$.
 Now assume $u \in W^k_0(D)$ for some $k \leq s-\frac{1}{2}$.
 Let $u=Pu+u^{\bot}$ be the decomposition of $u$ into a holomorphic function and the orthogonal part.
 Then there is 
 $$
 \Delta u^{\bot}=g-|\nabla f|^2 u,\;
 Pg=P(|\nabla f|^2 u).
 $$
 By Theorem \ref{sec5:thm-1}, we have
 $$
 u^{\bot}=N_{(0,0)}(g-|\nabla f|^2 u)\in W_0^{k+\frac{1}{2}}(D).
 $$ 
 On the other hand, we have
 $$
P(|\nabla f|^2 Pu)+P(|\nabla f|^2 u^{\bot})= P(|\nabla f|^2 u)=Pg \in W^s_0(D).
 $$
 Then it follows that
 $$
 P(|\nabla f|^2 u^{\bot}),P(|\nabla f|^2 Pu)  \in W_0^{k+\frac{1}{2}}(D).
 $$ 
 Now we have 
 \begin{eqnarray*}
 |||\nabla f|^2 Pu||_{k+\frac{1}{2}}&=&||P(|\nabla f|^2 Pu)+ \bar\partial^* N_{(0,1)}\bar{\partial} (|\nabla f|^2 Pu)||_{k+\frac{1}{2}}\\
 &\leq& ||P(|\nabla f|^2 Pu)||_{k+\frac{1}{2}} + ||\bar\partial^* N_{(0,1)}\bar\partial (|\nabla f|^2 Pu)||_{k+\frac{1}{2}}\\
 &\lesssim& ||P(|\nabla f|^2 Pu)||_{k+\frac{1}{2}} + ||\bar\partial (|\nabla f|^2 Pu)||_k\\
 &=& ||P(|\nabla f|^2 Pu)||_{k+\frac{1}{2}} + ||Pu\, \bar\partial |\nabla f|^2||_k\\
 &\lesssim& ||P(|\nabla f|^2 Pu)||_{k+\frac{1}{2}} + ||Pu||_k\\
 &\lesssim& ||P(|\nabla f|^2 Pu)||_{k+\frac{1}{2}} + ||u||_k,
 \end{eqnarray*}
 which gives 
 $$
 |\nabla f|^2 Pu \in W_0^{k+\frac{1}{2}}(D).
 $$
 Then there is
 $$
 |\nabla f|^2 u = |\nabla f|^2 Pu +|\nabla f|^2 u^{\bot} \in W_0^{k+\frac{1}{2}}(D).
 $$
 Since $\Delta + |\nabla f|^2$ is elliptic in the interior of $D$ and $|\nabla f|^2$ is nonzero on the boundary $\partial D$, we can apply elliptic estimate in the interior and divide by $|\nabla f|^2$ near the boundary to conclude that $u \in W_0^{k+\frac{1}{2}}(D)$.
 Now by induction, $u \in W^s_0(D)$ holds.
 \end{proof}
 
 \begin{rem}
 We can NOT expect that $u$ has higher regularity than $g$, which will lead to the compactness of $G_0$ by Rellich's lemma.
 For example, when $|\nabla f|^2$ happens to be a positive constant $c$, which is the case when $f= \sum^n_{i=1}z^i$, $G_0$ will
 have a non-zero eignvalue $\frac{1}{c}$ and the infinite dimensional Bergman space as the corresponding eignspace.
 Thus $G_0$ can not be compact and that's why we use an indirect way to prove the strong decomposition theorem.
 \end{rem}

 \begin{pro}\label{sec5:pro-1}
 Assume $D$ is strongly pseudoconvex with smooth $\partial D$ and $f$ has no critical points on $\partial D$. Then the Green operator $G$ is exactly regular and ${\Hcall}^* \subset \A^*(\bar D)$.
 \end{pro}
 
 \begin{proof}
 Assume $\Delta_f \varphi =\psi$ and $\psi \in W^s_p(D)$.\
 
 For $0 \leq p \leq n$, every $\varphi \in L^2_p(D)$ can be decomposed by types as $\varphi= \varphi^{p,0}+\varphi^{'}$.
 We have
 $$\Delta \varphi^{'}= \psi^{'}-|\nabla f|^2 \varphi^{'}- (L_f \varphi)^{'} \in L^2_p(D)$$
 According to Theorem \ref{sec5:thm-1}, $\varphi^{'} \in W^1_p(D)$.
 As $L_f$ always transform $(p,q)$-forms into sum of $(p-1,q+1)$-forms and $(p+1,q-1)$-forms, $\varphi^{p,0}$ does not contribute to the $(p,0)$ component of $L_f \varphi$. Therefore
 \begin{eqnarray*}
 \Delta \varphi^{p,0}+|\nabla f|^2 \varphi^{p,0}&=&\psi^{p,0}-(L_f \varphi)^{p,0}\\
 &=&\psi^{p,0}-(L_f \varphi^{'})^{p,0} \in W^1_p(D)
 \end{eqnarray*}
 Since the Neumann boundary condition for $(p,0)$-forms are the same as that for functions, so by Lemma \ref{sec5:lem-1} $\varphi^{p,0}\in W^1_p(D)$.
 Now by induction, exact regularity for $G_p$ holds.\
 
 For $n < p \leq 2n$, as no $(k,0)$-forms are involved, by Theorem \ref{sec5:thm-1} (4), the proof follows as the standard bootstrap argument.\
 
 Finally, let $\psi=0$, then Sobolev's imbedding theorem guarantees that forms in $\Hcall^*$ are all smooth up to the boundary.
 \end{proof}

 \section{Dimension computation}

 In this section, we will compute the dimension of the $L^2$ cohomology group $H^p_{((2),\bpat_f)}$ for $0\le p\le n$.

\begin{pro}\label{sec4:prop-1}
The complex $(\A^*(\bar D),\bpat_f)$ is quasi-isomorphic to the $L^2$ complex $(L^2(D),\bpat_f)$.
\end{pro}

\begin{proof}
For any $\phi\in \A^*(\bar D), \bpat_f \phi=0$, we can solve the $\bpat_f$-Neumann problem to obtain:
$$
\phi=P\phi+\bpat_f\bpat_f^*G\phi.
$$
By Proposition \ref{sec5:pro-1}, $P\phi \in \A^*(\bar D)$ and $\bpat_f^*G\phi\in \A^{*-1}(\bar D)$. This gives an isomorphism
$$
H^*(\A^*(\bar D),\bpat_f)\cong \Hcall^*\cong H^*_{((2),\bar{\partial}_f)}(D).
$$
\end{proof}

If we ignore the $L^2$ condition, we have the smooth complex $(\A^*(D),\bpat_f)$, whose cohomology can be computed by the spetral sequence as follows.

 This complex can be viewed as a total complex of the double complex $(\A^{*,*}(D),\bar{\partial},\partial f \wedge)$ with horizontal operator $\bar{\partial}$ and vertical operator $\partial f \wedge$. We consider the spectral sequence associated to the filtration
 $$
 {\Fcall}^k \A(D)= \bigoplus_{i \geq k} \A^{i,*}
 $$
 for $k \in \mathbb{Z}$.
 Since $D$ is pseudo-convex, it is stein, therefore the first page is concentrated at the first column with $(k,0)$ term given by holomorphic $k$-forms. Since the $f_i$'s have only finite common zeros, they form a regular sequence on $D$. Therefore the cohomology of the holomorphic Koszul complex
 $$0 \rightarrow \Omega^0(D) \xlongrightarrow{\partial f \wedge} \Omega^1(D) \xlongrightarrow{\partial f \wedge} \cdots \xlongrightarrow{\partial f \wedge} \Omega^n(D) \xlongrightarrow{\partial f \wedge} 0$$
 which is $E_2$, is concentrated at the top term $\Omega^n(D)/ \partial f \wedge \Omega^{n-1}(D) \cong \Jac(f)$. Thus the spectral sequence degenerate at the $E_2$-stage and the cohomology of the smooth $\bar{\partial}_f$ complex is concentrated at the middle dimention and is isomorphic to $\Jac(f)$. Hence we obtain the following result.

 \begin{pro}\label{sec4:prop-2}
 \begin{equation}
H^k_{\bpat_f}(D)=
\begin{cases}
\Jac(f)&k=n\\
0&k\neq n.
\end{cases}
\end{equation}
\end{pro}

To construct the cohomology of the complex $(\A^*_c(D),\bpat_f)$ consisting of the forms with compact support, we want to use    a homotopy introduced in \cite{LLS}.

 Let $\rho$ be a smooth function with compact support in $D$ such that it equals to $1$ in a neighborhood of $Crit(f)$. Define the following operator
 $$V_f= \sum^n_{i=1} \frac{\bar{f_i}}{|\nabla f|^2} (dz_i \wedge)^* : \quad \A^{*,*}(D \setminus Crit(f)) \rightarrow \A^{*-1,*}(D \setminus Crit(f))$$

A direct calculation gives the following result.
\begin{lem}\label{sec4:lemm-1}
\begin{equation}
 [df \wedge, V_f]=1
 \end{equation}
 and
 \begin{equation}
 [\bar{\partial},[\bar{\partial},V_f]]=[df \wedge,[\bar{\partial},V_f]]=[V_f,[\bar{\partial},V_f]]=0
 \end{equation}
 \end{lem}

Define two operators on $D$:
\begin{align}
T_{\rho}= \rho + (\bar{\partial} \rho)V_f \frac{1}{1+[\bar{\partial},V_f]},\;
R_{\rho}= (1-\rho)V_f \frac{1}{1+[\bar{\partial},V_f]}
\end{align}

\begin{lem}\label{sec4:lemm-2}
\begin{equation}
[\bar{\partial}_f,R_{\rho}]=1-T_{\rho} \quad on \quad \A^*(D)
\end{equation}
\end{lem}

\begin{proof}
 By Lemma \ref{sec4:lemm-1},
 \begin{eqnarray*}
 [\bar{\partial}_f,R_{\rho}]&=& [\bar{\partial}_f,1-\rho]V_f \frac{1}{1+[\bar{\partial},V_f]}
 +(1-\rho)[\bar{\partial}_f,V_f] \frac{1}{1+[\bar{\partial},V_f]}\\
 &=& (-\bar{\partial} \rho)V_f \frac{1}{1+[\bar{\partial},V_f]}+ (1-\rho)(1+[\bar{\partial},V_f]) \frac{1}{1+[\bar{\partial},V_f]}\\
 &=& 1-T_{\rho} \quad on \quad \A^*(D)
 \end{eqnarray*}
 \end{proof}

The Lemma \ref{sec4:lemm-2} built a homotopy from $(\A(D),\bar{\partial}_f)$ to $(\A_c(D),\bar{\partial}_f)$, hence we have

 \begin{pro}\label{sec4:prop-3}
\begin{equation}
H^*_{\bpat_f}(D)\cong H^*_{(c,\bpat_f)}(D).
\end{equation}
 \end{pro}

\begin{proof}Let $i:\A_c^k(D)\to \A^k(D)$ be the inclusion. Since $\bpat_f(\A_c(D))\subset \bpat_f(\A(D))$, we have the well-defined homomorphism $i_*:H^k_{(c,\bpat_f)}(D)\to H^k_{\bpat_f}(D)$. Assume that $[i(b)]=0\in H^k_{\bpat_f}$, then there exists a $c\in \A^{k-1}(D)$ such that $b=\bpat_f c$. By Lemma \ref{sec4:lemm-2}, we have
$$
b=\bpat_f(T_\rho c+R_\rho\bpat_f c)=\bpat_f(T_\rho+R_\rho b),
$$
where $T_\rho+R_\rho b\in \A_c(D)$. This shows that $[b]$ is the zero class in $H^k_{(c,\bpat_f)}(D)$. Hence $i_*$ is injective. On the other hand, if $\bpat_f b=0$ for $b\in \A^k(D)$, then $b=T_\rho b+\bpat_f R_\rho b$, which shows that $i_*$ is also surjective.
\end{proof}

 Let us check $R_{\rho}$ more carefully.
 In a small neighborhood of $Crit(f)$, $R_{\rho}=0$.
 Outside such a neighborhood,
 $$R_{\rho} = (1-\rho)V_f \sum^n_{k=1}(-1)^k[\bar{\partial},V_f]^k$$
 Here $V_f$ is of order $0$ and
 \begin{eqnarray*}
 [\bar{\partial},V_f]&=& \sum_{i,j} [\frac{\partial}{\partial \bar{z}_i} d\bar{z}_i \wedge,\frac{\bar{f_j}}{|\nabla f|^2} (dz_j \wedge)^*]\\
 &=& \sum_{i,j}\frac{\partial}{\partial \bar{z}_i}(\frac{\bar{f_j}}{|\nabla f|^2}) d\bar{z}_i \wedge (dz_j \wedge)^*
 \end{eqnarray*}
 is also of order 0. So $R_{\rho}$ is actually smooth and bounded, and it defines a bounded operator from $\A(\bar D)$ to itself. Now using the homotopy in Lemma \ref{sec4:lemm-2}, we can also have the following result.
 \begin{pro}\label{sec4:prop-4}
 \begin{equation}
H^* (\A(\bar D),\bar{\partial}_f) \cong H^*_{\bpat_f}(D).
\end{equation}
  \end{pro}

 Combining the results of Proposition \ref{sec4:prop-1}, \ref{sec4:prop-2}, \ref{sec4:prop-3} and \ref{sec4:prop-4}, we obtain Theorem \ref{sec1:theo-2}.

\begin{rem}
When $D$ is only pseudoconvex with smooth boundary, interior regularity of $\Delta_f$ shows $\Hcall^* \subset \A^*(D)$ and $G$ preserves $\A^*(D)$. Thus similar argument like Proposition \ref{sec4:prop-1} can be applied to show $H^*(L^2(D)\cap \A^*(D)) \cong \Hcall^*$. Moreover, like Proposition \ref{sec4:prop-4}, $T_{\rho}$ and $R_{\rho}$ can be used to give an isomorphism between $H^k(L^2(D)\cap \A^*(D))$ and $H^k_{(c,\bar\partial_f)}$. So Theorem \ref{sec1:theo-2} still holds.
\end{rem}

Likewise, $H^k_{(0,\bar\partial_f)} \cong H^k_{(c,\bar\partial_f)}$ by using $T_{\rho}$ and $R_{\rho}$. Thus by Proposition \ref{sec3:prop-2}, $H^n_{\bpat_{-f}}(\Ccal)$ and $H^n_{\bpat_f}(\bD)$ are isomorphic when $D$ is strongly pseudoconvex.
We can also obtain this result by proving the following duality theorem.
\begin{thm}We have the isomorphism
\begin{equation}
H^k_{\bpat_{-f}}(\Ccal)\cong (H^{2n-k}_{\bpat_f}(\bD))^*
\end{equation}
\end{thm}
\begin{proof}
The idea is to construct a pairing
$$\varphi,\psi \mapsto \int_D \varphi \wedge \psi$$
for $\varphi \in H^{2n-k}_{\bpat_f}(\bD)$ and $\psi \in H^k_{\bpat_{-f}}(\Ccal)$.
Almost the same proof as the Proposition 5.1.5 in \cite{FK}, except the arise of $\bar \partial f \wedge$, which gives a minus sign here, shows that this is indeed a pairing. To show it is non-degenerate, we only need Theorem \ref{sec3:theo-4} to take the role of the condition Z(q) in Proposition 5.1.5 in \cite{FK}.
\end{proof}

 \end{document}